\documentclass[12pt]{amsart}
\usepackage[colorlinks,linkcolor=blue,citecolor=blue,urlcolor=blue, pdfcenterwindow, pdfstartview={XYZ null null 1.2}, pdffitwindow, pdfdisplaydoctitle=true]{hyperref}
\usepackage[english]{babel}
\usepackage{array}
\usepackage{threeparttable}
\usepackage{rotating}
\usepackage[section]{placeins}
\usepackage{amssymb} 
\usepackage{mathrsfs}
\usepackage{enumerate}
\usepackage{color}
\usepackage{amsrefs}        
\usepackage[all]{xy}

\input xy
\xyoption{all}
\newtheorem{lemma}{Lemma}[section]
\newtheorem{lemm}[lemma]{Lemma}

\newtheorem{coro}[lemma]{Corollary}
\newtheorem{theo}[lemma]{Theorem}

\newtheorem{defi}[lemma]{Definition}
\newtheorem{exam}[lemma]{Example}

\newcommand{\Z}{\mathbb Z}
\newcommand{\R}{\mathbb R}
\newcommand{\Q}{\mathbb Q}

\begin{document}
\title{Groups with twisted ${\bf p}$-periodic cohomology}
\author{Guido Mislin}
\author{Olympia Talelli}
\address{Dept. of Mathematics ETH Z\"urich, Switzerland}    
\curraddr{Dept. of Mathematics, Ohio State University}
\email{mislin@math.ethz.ch}
\address{Dept. of Mathematics, University of Athens, Greece}
\email{otalelli@cc.uoa.gr}
\date{May 17, 2014}
\begin{abstract}
We give a characterization of groups with twisted $p$-periodic cohomology
in terms of group actions on mod $p$ homology spheres. An
equivalent algebraic characterization of such groups is also presented.
\end{abstract}
\maketitle
\section{Introduction}
We will be considering groups with twisted $p$-periodic cohomology ($p$ a prime) in the following sense.
Write $\hat{\Z}_{p}(\omega)$ for the group of $p$-adic integers, equipped with a $G$-action via a homomorphism $\omega: G\to \hat{\Z}_p^\times$. For $M$ a $\Z G$-module, we write $M_\omega$ for the
$\Z G$-module $M\otimes\hat{\Z}_{p}(\omega)$ with diagonal $G$ action.
\begin{defi}{\label{twisted}}
A group $G$ is said to have twisted $p$-periodic cohomology, if there is a $k>0$, a homomorphism
$\omega:G\to \hat{\Z}_p^\times$ and a cohomology class 
$e_\omega\in H^n(G,\hat{\Z}_{p}(\omega))$ for some $n>0$, such that 
$$e_\omega\cup - : H^i(G,M)\to H^{i+n}(G,M_\omega)$$
is an isomorphism for all $i\ge k$ and all $p$-torsion $\Z G$-modules $M$
of finite exponent.
In case the twisting $\omega$ can be chosen to be trivial, we say that
$G$ has $p$-periodic cohomology.
\end{defi}

By replacing $e_\omega$ with $e_\omega^2$ we see that for  
$G$ with twisted $p$-periodic cohomology one can assume, if one wishes to,
that the degree $n$ of the periodicity
generator is even. In case of a finite group $G$ we infer, by replacing $e_\omega$ by a suitable cup power, that if $G$ has twisted $p$-periodic cohomology, it also has $p$-periodic cohomology.
A classical theorem states that a finite group has $p$-periodic cohomology
if and only if all
its abelian $p$-subgroups are cyclic. Moreover, the finite groups
with $p$-periodic cohomology have the following characterization in terms of actions on $\Z/p\Z$-homology
spheres. 

\begin{theo}[Swan \cite{Swan}]\label{Swan} A finite group $G$ has $p$-periodic cohomology if and only if
there exists a finite, simply connected free $G$-$CW$-complex, which has the same $\Z/p\Z$-homology as
some sphere.
\end{theo}

Our goal is to find a similar characterization for arbitrary groups with (twisted) $p$-periodic 
cohomology.

\begin{defi} A $CW$-complex $X$ is called a $\Z/p\Z$-homology $n$-sphere, if $H_*(X,\Z/p\Z)\cong H_*(S^n,\Z/p\Z)$.
\end{defi}

In Section \ref{mainsection} we will prove the following generalization of Theorem \ref{Swan}.

\begin{theo}\label{main} A group $G$ has twisted $p$-periodic cohomology if and only
if there exist a simply connected $\Z/p\Z$-homology sphere $X$, which is a free
$G$-$CW$-complex satisfying $cd_{\Z/p\Z}(X/G)<\infty$.
\end{theo}

\smallskip
For the definition of the cohomological dimension $cd_{\Z/p\Z}$ of a space see Section \ref{dim}.

\smallskip
As we will see (cf. Section \ref{last}), there are groups which have twisted $p$-periodic cohomology
but which do not have $p$-periodic cohomology. For
groups with $p$-periodic cohomology we prove the following characterization.

\begin{theo}{\label{main2}} A group $G$ has $p$-periodic cohomology
if and only if
there exist a free $G$-$CW$-complex $X$ with homotopically trivial $G$-action such that
$X$ is a $\Z/p\Z$-homology sphere satisfying $cd_{\Z/p\Z}(X/G)<\infty$.
\end{theo}
\noindent

We will also be considering groups with $\Z/p\Z$-periodic cohomology in the following sense.

\begin{defi}
A group $G$ is said to have $\Z/p\Z$-periodic cohomology, if there is a cohomology class
$e\in H^n(G,\Z/p\Z)$ for some $n>0$ and an integer $k>0$, such that for every $\Z/p\Z[G]$-module $M$ the map
$$e\cup - : H^i(G,M)\to H^{i+n}(G,M)$$
is an isomorphism for all $i\ge k$. 
\end{defi}

The following is a simple observation.

\begin{lemm}\label{twistedmodp}
Suppose that $G$ has twisted $p$-periodic cohomology. Then $G$ has $\Z/p\Z$-periodic cohomology.
\end{lemm}

Indeed, if $e_\omega\in H^n(G,\hat{\Z}_p(\omega))$ gives rise to twisted periodicity
as above
and $e_\omega(p)\in H^{n}(G,(\Z/p\Z)_\omega)$ denotes the mod p reduction of $e_\omega$, then 
$G$ has $\Z/p\Z$-periodic cohomology with periodicity generator the $(p-1)$-fold
cup product $e:=e_\omega(p)^{p-1}\in H^{n(p-1)}(G,\Z/p\Z)$.

\smallskip
If $M$ is a fixed $\Z G$-module which is $p$-torsion of finite exponent $p^{k+1}$,
then the $p^k(p-1)$-fold twisted module
$$M_{\omega^{p^k(p-1)}}:=((\cdots(M_\omega)\cdots)_\omega)_\omega$$
is naturally isomorphic as a $\Z G$-module to $M$. Therefore, if $G$ has
twisted $p$-periodic cohomology of some period $n$,
its cohomology with $M$ coefficients will actually be periodic in high dimensions $d\ge d_0(M)$, with period $n\cdot p^k(p-1)$. In general, it is not possible to choose the dimensions $d_0(M)$ so that they are bounded
by a number independent of $M$. This observation leads to an example of a group with twisted $p$-periodic cohomology
but not having $p$-periodic cohomology (cf. Example \ref{mainexample}). 

\smallskip
It is this
example together with the fundamental paper by Adem and Smith \cite{AS} which inspired our work. 
For background on groups acting freely on finite dimensional homology spheres, see \cite{MT} and \cite{T}.

\smallskip
  
\section{$\Z/p\Z$-dimension for spaces and $\Z/p\Z$-localization}\label{dim}
Similarly to the definition of the $\Z/p\Z$-cohomological dimension of groups,
one defines the $\Z/p\Z$-cohomological dimension for spaces as follows.
\begin{defi}\label{cdim}
Let $X$ be a connected $CW$-complex and $k>0$. 
The $\Z/p\Z$-co\-ho\-mo\-logical dimension $cd_{\Z/p\Z}(X)$ of $X$ is the 
smallest integer~$n$
such that $H^i(X,M)=0$ for all $\Z/p\Z[\pi_1(X)]$-modules $M$ and all $i>n$; if
there is no such $n$, we write $cd_{\Z/p\Z}(X)=\infty$. 
\end{defi} 

A simple induction on $k$ shows that if $cd_{\Z/p\Z}X<\infty$ then there exists an $i>0$
such that for all $k$ and all $\Z/p^k\Z[\pi_1(X)]$-modules $M$,
$H^j(X,M)=0$ for all $j>i$.

\bigskip

Bousfield constructed
in \cite{B} on the homotopy category of CW-com\-plexes the localization with respect to $H_*(-,\Z/p\Z)$,
which we call the $\Z/p\Z$-localization and which consists of a functorial map 
$$c(X):X\to X_{\Z/p\Z}$$
which is characterized by the following universal property :
for every $H_*(-,\Z/p\Z)$-isomorphism $f:X\to Z$ there is a unique map (up to homotopy)
$g : Z\to X_{\Z/p\Z}$ which is an $H_*(-,\Z/p\Z)$-isomorphism such that $g\circ f\simeq c(X)$.

$\quad\quad\quad\quad\quad\quad\xymatrixrowsep{3pc}\xymatrixcolsep{3pc}\xymatrix{
X\ar[rr]^-{H_*(-,\Z/p\Z)-iso\  f}\ar[d]_{c(X)} && Z\ar[dll]^{\exists !\  g}\\
X_{\Z/p\Z}\,.\\
}$

\noindent
If $X$ is simply connected (or nilpotent) and of finite type, then $X_{\Z/p\Z}$ agrees with Sullivan's
$p$-completion $\hat{X}_p$ (cf.~\cite{S}), and $X\to \hat{X}_p$ is profinite $p$-completion on the level of
homotopy groups.
\smallskip

Note that if $X$ is simply connected, one has $cd_{\Z/p\Z}X=cd_{\Z/p\Z}X_{\Z/p\Z}$, but for instance
$cd_{\Z/p\Z}S^1=1<cd_{\Z/p\Z}S^1_{\Z/p\Z}=\infty$ (because $\pi_1(S^1_{\Z/p\Z})$ contains a free
abelian subgroup of infinite rank).
\smallskip

By the {\sl{standard}} $\Z/p\Z$-homology
$n$-sphere we mean $S^n_{\Z/p\Z}$. 

\begin{lemm}\label{Z/pZ-sphere} Let $X$ be a $\Z/p\Z$-homology $n$-sphere.
Then $X_{\Z/p\Z}$ is homotopy equivalent to $S^n_{\Z/p\Z}$.
\end{lemm}

\begin{proof}
Assume that
$H_*(X,\Z/p\Z)\cong H_*(S^n,\Z/p\Z)$. We first consider the case of $n=1$. It follows that
$\pi_1(X)_{ab}\otimes \Z/p\Z$ $\cong$ $\Z/p\Z$. Choose an $f:S^1\to X$ mapping to a generator of
$\pi_1(X)_{ab}\otimes \Z/p\Z$. Then $f$ induces an isomorphism in homology with $\Z/p\Z$-coefficients.
It follows that $f$ induces a homotopy equivalence $S^1_{\Z/p\Z}\to X_{\Z/p\Z}$.
Now assume that $n>1$. Since $H_1(X,\Z/p\Z)\cong H_1(X_{\Z/p\Z}, \Z/p\Z)=0$, we also have
$H_1(\pi_1(X_{\Z/p\Z}),{\Z/p\Z})$=$0$. But $\pi_1(X_{\Z/p\Z})$
is a $H\Z/p\Z$-local group, thus $\pi_1(X_{\Z/pZ})=0$ (see Theorem 5.5 of \cite{B}).
We proceed by showing that $X_{\Z/p\Z}$ is $(n-1)$-connected. Let $\pi_i(X_{\Z/p\Z})$ be the first
non-vanishing homotopy group of $X_{\Z/p\Z}$, $i>1$. Because a $H_*(-,\Z_{(p)})$-isomorphism is also
an $H_*(-,\Z/p\Z)$-isomorphism, $X_{\Z/p\Z}$ is $H\Z_{(p)}$-local and therefore its homology groups with
$\Z$-coefficients are uniquely $q$-divisible for $q$ prime to $p$. Moreover, for $n>i>1$, multiplication by $p$
is bijective on $H_i(X_{\Z/p\Z},\Z)$, because $H_j(X_{\Z/p\Z},\Z/p\Z)=0$ for $j=i-1,i$. Thus $H_i(X_{\Z/p\Z},\Z)$
is a $\Q$-vector space for $1<i<n$. Since the only $\Q$-vector space, which is $H\Z/p\Z$-local as an abelian group,
is the trivial one, and because the homotopy groups of $X_{\Z/p\Z}$ are $H\Z/p\Z$-local, we conclude from the
Hurewicz Theorem that $X_{\Z/p\Z}$ must be $(n-1)$-connected. It follows that the natural maps
$$\pi_n(X_{\Z/p\Z})\to H_n(X_{\Z/p\Z},\Z)\to H_n(X_{\Z/p\Z},\Z,p\Z)\cong \Z/p\Z$$
are both surjective. Choose an $f: S^n\to X_{\Z/p\Z}$ which maps to a generator of $H_n(X_{\Z/p\Z},\Z/p\Z)$ and it
follows that $f$ induces a homotopy equivalence $S^n_{\Z/p\Z}\to X_{\Z/p\Z}$. 
\end{proof}
\smallskip
There is also a fiberwise version of $\Z/p\Z$-localization (see \cite{M} for details). If 
$$X\to E\to B$$ is a fibration
of connected CW-complexes, one can construct a new fibration
$$X_{\Z/p\Z}\to E^f_{\Z/p\Z}\to B\,,$$ together with a map $E\to E^f_{\Z/p\Z}$ over $B$, which restricts on
the fibers to $\Z/p\Z$-localization $X\to X_{\Z/p\Z}$. 
Using the Serre spectral sequence, we conclude the following. If $F\to E\to B$ is a fibration
of connected $CW$-complexes with $F$ simply connected, then $cd_{\Z/p\Z}E=cd_{\Z/p\Z}E^f_{\Z/p\Z}$.
Also, if the fiber $F$ is a $\Z/p\Z$-homology sphere, then fiberwise $\Z/p\Z$-localization
yields a fibration with fiber a standard $\Z/p\Z$-homology sphere $S^n_{\Z/p\Z}\to E^f_{\Z/p\Z}\to B.$

\section{Fibrations, orientation and Euler class}

If $F\to E\to B$ is a fibration of connected $CW$-complexes, then $\pi_1(E)\to \pi_1(B)$
is surjective and lifting of loops defines a natural map $\theta: \pi_1(B)\to [F,F]$, a 
{\sl homotopy action} of $\pi_1(B)$ on $F$.

\begin{defi}\label{orientable}
Let $F\to E\to B$ be a fibration of connected CW-complexes. The fibration is
called orientable, if the associated homotopy action $\pi_1(B)\to [F,F]$ is
trivial. We call the fibration $H\Z/p^k\Z$-orientable, if $\pi_1(B)$ acts trivially
on $H_*(F,\Z/p^k\Z)$.
\end{defi}
Clearly, if a fibration is orientable, it is $H\Z/p^k\Z$-orientable for all $k$.

\begin{defi}\label{R-spherical}
Let $
 F\to E\to B$
be a fibration of connected CW-complexes. 
We call such a fibration {\sl{$\Z/p\Z$-spherical}}  
in case $F$ is a $\Z/p\Z$-homology
sphere (or, equivalently, if $F_{\Z/p\Z}\simeq S^n_{\Z/p\Z}$ for some $n>0$).
\end{defi}

We will make use of the following   
observation.

\begin{lemm}\label{spacetofibration}
For a group $G$ the following are equivalent.
\begin{itemize}
\item[a)] There exists a simply connected free $G$-$CW$-complex $X$ which is a
$\Z/p\Z$-homology sphere satisfying $cd_{\Z/p\Z}X/G<\infty$.
\item[b)] There exists a $\Z/p\Z$-spherical fibration 
$F\to E\to K(G,1)$ with $F$ simply connected and $cd_{\Z/p\Z}E<\infty$.
\end{itemize}
\end{lemm}

\begin{proof}
Let $X$ be as in $a)$ and $f: X/G\to K(G,1)$ the classifying map for the universal cover $X$
of $X/G$. Then the homotopy fiber of $f$ is $G$-homotopy equivalent to $X$, thus $b)$ holds.
If $F\to E\to K(G,1)$ is as in $b)$, the universal cover of $E$ is $G$-homotopy equivalent to $F$, thus $a)$ holds. 
\end{proof}

Note that if $X$ any $\Z/p\Z$-homology sphere, it is also a $\Z/p^k\Z$-homo\-lo\-gy sphere 
for $k>1$ as 
one easily sees by induction on $k$.
Thus, for a $\Z/p\Z$-spherical fibration 
$$F\to E\to B$$ 
as in Definition \ref{R-spherical},
the $\pi_1(B)$-module $H_n(F,\Z/p^k\Z)$ 
is isomorphic to a twisted module $(\Z/p^k\Z)_\omega$, where
$\omega: \pi_1(B)\to (\Z/p^k\Z)^\times$ corresponds to the action of $\pi_1(B)$ on $H_n(F,\Z/p^k\Z)$.
(If we need to emphasize the dependence of $\omega$ on $k$, we write $\omega(k)$ in place of $\omega$).
We call the twisted module $(\Z/p^k\Z)_\omega$ the $k$-orientation module.
The fibration is {{H$\Z/p^k\Z$-orientable}}
in the sense of Definition \ref{orientable}, if the $k$-orientation
module is the trivial $\Z/p^k\Z[\pi_1(B)]$-module $\Z/p^k\Z$. We write ${\bar{\omega}}$ for the map
$\pi_1(B)\to (\Z/p^k\Z)^\times$ given by ${\bar{\omega}}(x)=\omega(x^{-1})$,
and more generally $\omega^n$, for the map with $\omega^n(x)=\omega(x^n)$, $n\in\Z$.
For any $\Z/p^k\Z[\pi_1(B)]$-module $M$ we write $M_\omega$ for $M\otimes (\Z/p^k\Z)_\omega$
with diagonal $\pi_1(B)$-action $x\cdot (m\otimes z)=xm\otimes\omega(x)z$. Similarly, we consider
the diagonal action on $Hom_{\Z/p^k\Z}((\Z/p^k\Z)_\omega),M)$ given by $(xf)(z)=x\cdot f(\bar{\omega}(x)z)$.
Therefore, there is a natural isomorphism of $\Z/p^k\Z[\pi_1(B)]$-modules
$$H^n(F,M)\cong Hom(H_n(F,\Z/p^k\Z),M)\cong Hom((\Z/p^k\Z)_\omega,M)
\cong M_{\bar{\omega}}\,.$$
In the case of a $\Z/p\Z$-spherical fibration $F\to E\to B$, 
the only possibly non-zero differential in the Serre spectral sequence
with coefficients in a $\Z/p^k\Z[\pi(E)]$-module $K$,

$$E_2^{i,\ell}=H^i(B,H^\ell(F,K))\Longrightarrow H^{i+\ell}(E,K),$$
is the transgression differential
$$d_{n+1}: E^{i,n}_2=E^{i,n}_{n+1}\to E^{i+n+1,0}_{n+1}=E^{i+n+1,0}_2.$$
Taking for $K$ the $k$-orientation module $(\Z/p^k\Z)_\omega$
and choosing $i=0$, this yields
$$d_{n+1}: \Z/p^k\Z=H^0(B,\Z/p^k\Z)\to H^{n+1}(B,(\Z/p^k\Z)_{\omega})\,,$$
and the image of $1\in \Z/p^k\Z$, $d_{n+1}(1):=e(k)_\omega\in H^{n+1}(B, (\Z/p^k\Z)_\omega)$
is called the {\sl{twisted}} $\Z/p^k\Z$-Euler class of the given $\Z/p\Z$-spherical
fibration. 
Let now $M$ be an arbitrary $\Z/p^k\Z[\pi_1(B)]$-module and choose $K=M_\omega$. Thus $H^n(F,M_\omega)=M$
and  
$$d_{n+1}: E^{i,n}_2=H^i(B,M)\to H^{i+n+1}(B,M_\omega)=E^{i+n+1,0}_2\,$$
is given by the cup-product with $e(k)_\omega$. The kernel and image of $d_{n+1}$ are determined as follows:
$$E^{i,n}_\infty=\ker{d_{n+1}}\subset E^{i,n}_2=H^i(B,M)\stackrel{d_{n+1}}\longrightarrow H^{i+n+1}(B,M_\omega)
\,,$$
respectively
$$
H^i(B,M)\stackrel{d_{n+1}}\longrightarrow H^{i+n+1}(B,M_\omega)=E^{i+n+1,0}_2
\twoheadrightarrow\operatorname{coker}{d_{n+1}}=E^{i+n+1,0}_\infty. $$
The natural surjection $\sigma: H^{i+n}(E,M)\to E^{i,n}_\infty$ has as kernel the subgroup
$E^{i+n,0}_\infty$ and, by splicing things together one gets the {\sl{Gysin-sequence}}
$$\to\! H^{i+n}(E,M)\!\overset{\sigma}\rightarrow\! H^i(B,{M}\,)\!\stackrel{e(k)_\omega\cup -}\longrightarrow\!
H^{i+n+1}(B,M_\omega)\!\to \!H^{i+n+1}(E,M_\omega)\!\to$$
One concludes that for large values of $i$ and all $\Z/p^k\Z[\pi_1(B)]$-modules $M$, cup product with $e(k)_\omega$
induces for all $k$ isomorphisms
$$e(k)_\omega\cup - : H^i(B,M)\overset{\cong}\longrightarrow H^{i+n+1}(B,M_\omega)$$
if and only there exists a $j_0$ such that for all $j>j_0$, $H^j(E,M)=0$ for all $\Z/p^k\Z[\pi_1(B)]$-modules $M$
and all $k$ (here $M$ is viewed as $\pi_1(E)$-modules via $\pi_1(E)\to \pi_1(B)$).
In case $F$ is simply connected, this amounts to $cd_{\Z/p\Z}E<\infty$.

\begin{coro}{\label{c1}}
Let $F\to E\to B$ be a $\Z/p\Z$-spherical fibration of $CW$-complexes with $B$ connected and $F$
simply connected, with twisted $\Z/p^k\Z$-Euler classes
$e(k)_{\omega(k)}\in H^{n}(B,(\Z/p^k\Z)_{\omega(k)})$, $k\ge 1$. Then
the following are equivalent.
\begin{itemize}
\item[1)] $cd_{\Z/p\Z}E<\infty$;
\item[2)] there exists $i_0$ such for all $i>i_0$ and all $k\ge 1$
$$e(k)_{\omega(k)}\cup - : H^i(B,M)\longrightarrow H^{i+n}(B,M_{\omega(k)})$$
is an isomorphism for all $\Z/p^k\Z[\pi_1(B)]$-modules $M$.
\end{itemize}
\end{coro}

In the situation of Corollary \ref{c1}, it follows from the naturality of the Serre
spectral sequence that the twisted $\Z/p^k\Z$-Euler classes $e(k)_{\omega(k)}$
are the reduction mod $p^k$ of a class $e_\omega\in H^n(B,\hat{\Z}_p(\omega))$,
where $\hat{\Z}_p(\omega)$ is isomorphic to $\pi_{n-1}(F_{\Z/p\Z})\cong\pi_{n-1}(S^{n-1}_{\Z/p\Z})$
as a $\pi_1(B)$-module.
Therefore, the following holds.

\begin{coro}\label{fibration to twisted}
If there exists a $\Z/p\Z$-spherical fibration of $CW$-com\-plexes $F\to E\to K(G,1)$ with $F$
simply connected and $cd_{\Z/p\Z}E<\infty$, then $G$ has twisted $p$-periodic cohomology.
\end{coro}

The following lemma permits us to pass from $\Z/p\Z$-spherical fibrations to $H\Z/p\Z$-orientable ones.

\begin{lemm}\label{c2} Let $F_1\to E_1\to B$ be a $\Z/p\Z$-spherical fibration of $CW$-complexes with $B$ connected and $F_1$ simply connected, such that $cd_{\Z/p\Z}E_1<\infty$.
Then the $(p-1)$-fold fiberwise join yields a $H\Z/p\Z$-orientable $\Z/p\Z$-spherical fibration
$F_2\to E_2\to B$ over the same base, with $cd_{\Z/p\Z}E_2<\infty$.
\end{lemm}

\begin{proof}
Let $e_\omega\in H^n(B,(\Z/p\Z)_\omega)$ be the twisted Euler class of the fibration $F_1\to E_1\to B$.
Because $cd_{\Z/p\Z}E_1<\infty$, we infer from Corollary \ref{c1} that
there exists $i_0$ such that
$$e_\omega\cup - : H^i(B, M)\to H^{i+n}(B,M_\omega)$$
is an isomorphism for all $i>i_0$ and all $\Z/p\Z[\pi_1(B)]$-modules $M$. We then perform a fiberwise $(p-1)$-fold join 
to obtain a new $\Z/p\Z$-spherical fibration
$F_2\to E_2\to B$ with Euler class $e=e_\omega^{p-1}$. This new
fibration is $H\Z/p\Z$-orientable, because the $(p-1)$-fold tensor product of $(\Z/p\Z)_\omega$ 
with diagonal action is the
trivial $\Z/p\Z[\pi_1(B)]$-module $\Z/p\Z$. Moreover
$$e\cup -: H^i(B,M)\to H^{i+(p-1)n}(B,M)$$
is an isomorphism for $i>i_0$ and all $\Z/p\Z[\pi_1(B)]$-modules $M$. 
Note that $e$ is the reduction mod $p$ of the twisted $\Z/p^k\Z$-Euler class
$e(k)_{\omega(k)}\in H^n(B,(\Z/p^k\Z)_{\omega(k)})$ of the $\Z/p\Z$-spherical fibration
$F_2\to E_2\to B$. Induction on $k$ then shows that
$$e(k)_{\omega(k)}\cup -: H^i(B,L)\to H^i(B,L_{\omega(k)})$$
is an isomorphism for all $\Z/p^k\Z[\pi_1(B)]$-modules $L$.
We infer from 
Corollary \ref{c1} that $cd_{\Z/p\Z}E_2<\infty.$
\end{proof}

\section{Partial Euler classes}

For a connected $CW$-complex $X$ we write $P_qX$ for its
$q$-th Postnikov section, with canonical map $X\to P_qX$ such that
\begin{itemize}\item[(1)]
$\pi_i(P_q(X))=0$ for $i>q\,,$
\item[(2)] $\pi_j(X)\stackrel{\cong}\longrightarrow\pi_j(P_qX)$ for $j\le q\,.$
\end{itemize}
\smallskip\noindent
In case that $X$ is a $\Z/p\Z$-homology sphere, we have $X_{\Z/p\Z}\simeq S^n_{\Z/p\Z}$.
Therefore, $P_q(X_{\Z/p\Z})=\{\ast$\} for $q<n$ and $P_n(X_{\Z/p\Z})\simeq K(\hat{\Z}_p,n)$. 
Adapting the terminology of \cite{AS} we define {\sl k-partial $\Z/p\Z$-Euler classes} as follows.
\begin{defi}\label{kpartial}
Let $B$ be a connected $CW$-complex and and $k\ge 0$. Then $\epsilon\in H^n(B,\Z/p\Z)$ is
a {\sl k-partial $\Z/p\Z$-Euler class} if there exists a fibration
$$(\Phi):\quad\quad P_{n-1+k}(S^{n-1}_{\Z/p\Z})\to E\to B$$
such that $\pi_1(B)$ acts trivially on $H^{n-1}(P_{n-1+k}(S^{n-1}_{\Z/p\Z}),\Z/p\Z)\cong \Z/p\Z$
and there is a generator of that group which transgresses to $\epsilon$
in the Serre spectral sequence with $\Z/p\Z$-coefficients for the fibration $(\Phi)$.
The $k$-partial $\Z/p\Z$-Euler class $\epsilon$ is called {\sl{orientable}},
if the fibration $(\Phi)$ can be chosen to be orientable in the sense of Definition \ref{orientable}.
\end{defi}
\begin{lemm}\label{powers}
Let $B$ be a connected $CW$-complex and $\epsilon\in H^n(B,\Z/p\Z)$ a $k$-partial $\Z/p\Z$-Euler class.
Then for all $\ell>0$, $\epsilon^\ell$ is a $k$-partial $\Z/p\Z$-Euler class. If $\epsilon$ is orientable
in the sense of Definition \ref{kpartial}, then so is
$\epsilon^\ell$.
\end{lemm}
\begin{proof}
Let 
$P:= P_{n-1+k}(S^{n-1}_{\Z/p\Z})\to E\to B$
be a fibration such that $\pi_1(B)$ acts trivially on $H^{n-1}(P,\Z/p\Z)$ and let $\alpha\in
H^{n-1}(P,\Z/p\Z)$ be an element which transgresses to $\epsilon$. By forming fiberwise the $\ell$-fold join
and applying $\Z/p\Z$-localization,
we obtain a new fibration
$(\ast^\ell P)_{\Z/p\Z}\to E(\ell)\to B\,.$
In the Serre spectral sequence with $\Z/p\Z$-coefficients for this fibration, $(\alpha\ast\cdots\ast\alpha)_{\Z/p\Z}$
transgresses to $\epsilon^\ell$. Since $\ast^\ell S^{n-1}\simeq S^{n\ell-1}$,
$$P_{n\ell-1+k}((\ast^\ell P)_{\Z/p\Z})=P_{n\ell-1+k}(S^{n\ell-1}_{\Z/p\Z})$$
and we obtain by taking fiberwise Postnikov sections a fibration
$$P_{n\ell-1+k}(S^{n\ell-1}_{\Z/p\Z})\to E^f(\ell)\to B$$ for which the image of $(\ast^\ell\alpha)_{\Z/p\Z}$ under the natural map
$$ H^{n\ell-1}((\ast^\ell P)_{\Z/p\Z},\Z/p\Z)\stackrel{\cong}\longrightarrow H^{n\ell-1}(P_{n\ell-1+k}(S^{n\ell-1}_{\Z/p\Z}),\Z/p\Z)$$
transgresses to $\epsilon^\ell$. It is obvious that $\epsilon^\ell$ is orientable if $\epsilon$ is.
\end{proof}

\begin{lemm}\label{P10}
Let 
$(\Phi_0): S^n_{\Z/p\Z}\to E \to B$
\noindent
be a fibration with $B$ connected and $n>0$. By taking fiberwise Postnikov sections, we obtain fibrations
$$(\Phi_k):\quad P_{n+k}S^n_{\Z/p\Z} \to E_k \to B\,,\quad k\ge 0.$$
\noindent
The fibrations $(\Phi_k)$, $k\ge 0$, are all orientable 
if and only if
$\pi_1(B)$ acts trivially on $\pi_n(S^n_{\Z/p\Z})\cong \hat{\Z}_p$.
\end{lemm}

\begin{proof} This follows from the functoriality of $P_{n+k}$
and the fact that homotopy classes $S^n_{\Z/p\Z}\to S^n_{\Z/p\Z}$
correspond naturally to elements of $\pi_n(S^n_{\Z/p\Z})$.
\end{proof}

\begin{defi}\label{omega-p-integral}
Let $X$ be a connected $CW$-complex with fundamental group $G$.
We call an element $x\in H^n(X,\Z/p\Z)$ $\omega$-p-integral, if there
exists an action $\omega: G\to {\hat{\Z}_p}^\times$ such that $G$ acts
trivially on $\hat{\Z}_p/p\hat{\Z}_p$ and $x$ lies
in the image of the natural coefficient homomorphism
$H^n(X,\hat{\Z}_p(\omega))$ $\to H^n(X,\Z/p\Z)\,.$ In case the action
$\omega$ can be chosen to be trivial, we call $x$ p-integral.
\end{defi}

To deal with non-orientable fibrations, we recall
the following fact. Let
$$\hspace{-2 cm} (F):\hspace{2 cm} K(M,m)\to E \to B$$ 
be a fibration with connected base $B$, $m>0$ and induced action of $\pi_1(B)=G$ on $M$ 
corresponding to the homomorphism
$\phi:G\to Aut(M)\,.$
Such fibrations are classified by cohomology elements with local coefficients as follows. There 
is a universal fibration
$$K(M,m+1)\to L_\phi(M,m+1)\to K(G,1)$$
such that fibrations of type $(F)$ 
correspond to homotopy classes of maps $f:B\to L_\phi (M,m+1)$ over $K(G,1)$.
The homotopy class over $K(G,1)$ of such an $f$ corresponds to an element in the cohomology with local coefficients $H^{m+1}(B,M)$,
see \cite{Baues} or \cite{G}.

The following lemma is a variation of Lemma 2.5 of \cite{AS}.
\begin{lemm}\label{beta}
Let  
$x\in H^{2n}(X,\Z/p\Z)$ be an $\omega$-p-integral element. Then some cup power of $x$
is a $k$-partial $\Z/p\Z$-Euler class and this $k$-partial $\Z/p\Z$-Euler class
is orientable (in the sense of Definition \ref{kpartial}) in case $x$ is p-integral.
\end{lemm}
 
\begin{proof} Let $G$ be the fundamental group of $X$.
Since $x$ is $\omega$-$p$-integral, 
there exists $\omega: G\to\hat{\Z}_p^\times$ and $\tilde{x}\in H^{2n}(X,\hat{\Z}_p(\omega))$
mapping to $x$ under restriction mod $p$.
Let $\mu: X\to L_\omega(\hat{\Z}_p,2n)$ correspond to $\tilde{x}$. It
classifies a fibration
$$K(\hat{\Z}_p(\omega),2n-1)\to E\to X$$
with $H^{2n-1}(K(\hat{\Z}_p(\omega),2n-1),\Z/p\Z)$ =  $H^{2n-1}(K(\hat{\Z}_p/p\hat{\Z}_p,2n-1),\Z/p\Z)$ $\cong$  $\Z/p\Z$
having trivial $G$-action. This shows that $x$ is a $0$-partial $\Z/p\Z$-Euler class.
Suppose now that $k>0$ is given and that $x^m$ is a $(k-1)$-partial
$\Z/p\Z$-Euler class. Thus there is a fibration 
$$P_{2nm-1+k-1}(S^{2nm-1}_{\Z/p\Z})=:P(k-1)\to E(k-1)\to X$$
with a generator of $H^{2nm-1}(P(k-1),\Z/p\Z)^G=\Z/p\Z$ transgressing to $y:=x^m$. By Lemma \ref{powers},
for all $j$, $y^j$ is a $(k-1)$-partial $\Z/p\Z$-Euler class too. Thus there are
fibrations
$$P_{2nmj-1+k-1}(S^{2nmj-1}_{\Z/p\Z})=:Q(k-1)\to F(k-1)\to X$$
with a generator of $H^{2nmj-1}(Q(k-1),\Z/p\Z)^G=\Z/p\Z$ transgressing to $y^j=x^{mj}$. To show that 
for a suitable $j$,
the power 
$y^j$ gives rise to
a $k$-partial $\Z/p\Z$-Euler class, we need to check that the classifying map $\theta:Q(k-1)\to K(\pi,2nmj+k)$ for the fibration $Q(k)\to Q(k-1)$ factors through $F(k-1)$. Note that
$$\pi:=\pi_{2nmj+k-1}(Q(k))=\pi_{2nmj+k-1}(S^{2nmj-1}_{\Z/p\Z})$$
is a finite $p$-group on which $\pi_1(X)=G$ acts
via
$$\omega^{mj}:G\to\hat{\Z}_p^\times=HoAut(S^{2nmj-1}_{\Z/p\Z})\,.$$
We write $\underline{\pi}$ for $\pi$
with that action. 
Because of the naturality of the Postnikov section functor, the homotopy
fibration
$$Q(k)\to Q(k-1)\stackrel{\theta}\longrightarrow K(\underline{\pi},2nmj+k)$$
is compatible with the homotopy $G$-action via $\omega^{mj}$ on these spaces.
Therefore,
$$[\theta]\in H^{2nmj+k}(Q(k-1),\underline{\pi})$$
is $G$-invariant with respect to the diagonal $G$-action on this cohomology group.
In the Serre spectral
sequence for $Q(k-1)\to F(k-1)\to X$ with $\underline{\pi}$ coefficients
$$H^s(X, H^t(Q(k-1),\underline{\pi}))\Rightarrow H^{s+t}(F(k-1),\underline{\pi})$$
the cohomology class $[\theta]$ lies thus in 
$$E_2^{0,2nmj+k}=H^{2nmj+k}(Q(k-1),\underline{\pi})^G$$
and to show that it is the restriction of a class in the cohomology of $F(k-1)$
with $\underline{\pi}$-coefficients amounts to show that $[\theta]$ is a permanent cycle. 
The same argument as in \cite{AS}*{Lemma 2.5}
shows that this is the case for $j$ a large enough $p$-power.
It follows that some power of $x$ is $k$-partial $\Z/p\Z$-Euler class.
In case $x$ is $p$-integral, the argument shows that the $k$-partial
$\Z/p\Z$-Euler class we obtained is orientable.   \end{proof}   
                                                                                                                                                                                                                                      
\section{Proof of Theorems \ref{main} and \ref{main2}}\label{mainsection}
We will give the proof of Theorem \ref{main}. The proof of Theorem \ref{main2}
is analogous but simpler.

Suppose that $G$ has twisted $p$-periodic cohomology. Then there exists
an $\omega$-$p$-integral class 
$\epsilon\in H^{2n}(G,\Z/p\Z)$ and $\epsilon_\omega
\in H^{2n}(G,\hat{\Z}_p(\omega))$ whose reduction mod $p$ is $\epsilon$. By assumption,
there is and a $\ell_0>0$, such that
cup product with $\epsilon_\omega$ induces isomorphisms $H^i(G,M)\to H^{i+2n}(G,M_\omega)$ for all $i\ge \ell_0$
and all $p$-torsion $\Z G$-modules $M$ of finite exponent.
By Lemma \ref{beta} we can find a cup-power $\epsilon^m$ which
is an $\ell_o$-partial $\Z/p\Z$-Euler class.
Therefore, we
have a fibration
$$F(\ell_0):\quad P_{2nm-1+\ell_0}(S^{2nm-1}_{\Z/p\Z})\to E(\ell_0)\to K(G,1)\;,$$
with the property that a generator of $H^{2nm-1}(P_{2nm-1+\ell_0}(S^{2nm-1}_{\Z/p\Z}),\Z/p\Z)$
$\cong \Z/p\Z$
transgresses to $\epsilon^m$ in the Serre spectral sequence for $F(\ell_0)$. 
We want to show inductively that $\epsilon^m$ is a $k$-partial Euler class for all $k\ge\ell_0$.
Write $P(j)$ for $P_{2nm-1+j}(S^{2nm-1}_{\Z/p\Z})$.
We will inductively construct fibrations 
$$F(k):\quad P(k)\to E(k)\to K(G,1)$$
for $k>\ell_0$
with the property that a generator of $H^{2nm-1}(P(k),\Z/p\Z)$ transgresses to $\epsilon^m$.
To pass from $F(k-1)$ to $F(k)$ we
argue as follows. 
We have a diagram 
$${\xymatrix{
F(k-1): &P(k-1)\ar[r]&E(k-1)\ar[r]&K(G,1)\\
F(k): &P(k)\ar[u]\ar@{.>}[r]&E(k)\ar@{.>}[r]\ar@{.>}[u]&K(G,1)\ar[u]^{=}\\
}}
$$
in which the fibration $P(k)\to P(k-1)$ has fiber $K(\pi(\omega),2nm-1+k)$ and
is classified by a map
$$\theta: P(k-1)\to K(\pi(\omega),2nm+k),$$
where $\pi(\omega)$ stands for the finite $p$-group $\pi:=\pi_{2nm-1+k}(S^{2nm-1})\otimes \hat{\Z}_p$
with $G$-action induced by $\omega^m: G\to\hat{\Z}_p^\times\subset Aut(\pi_{2nm-1}(S^{2nm-1}_{\Z/p\Z}))$.
To construct the fibration $F(k)$ and the dotted arrows depicted above, we need to
show that $\theta$ factors through $E(k-1)$. This amounts to showing that
$[\theta]\in H^{2nm+k}(P(k-1),\pi(\omega))$ lies in the image of the restriction map
$$H^{2nm+k}(E(k-1),\pi(\omega))\to H^{2nm+k}(P(k-1),\pi(\omega))\,.$$
As argued in the proof of Lemma \ref{beta}, $[\theta]\in H^{2nm+k}(P(k-1),\pi(\omega))$
is $G$-invariant with respect to the diagonal $G$-action  via $\omega^m$ on this cohomology group.
The restriction map in question
corresponds to an edge homomorphism
in the Serre spectral sequence with $\pi(\omega)$-coefficients for the fibration
$P(k-1)\to E(k-1)\to K(G,1)$,
$$H^{2nm+k}\!(E(k-1),\!\pi(\omega))\!\twoheadrightarrow\! E^{0,2nm+k}_\infty\!\subset\! E^{0,2nm+k}_2\!=\!
H^{2nm+k}\!(P(k-1),\!\pi(\omega))^G\!.$$
We need therefore to check that $[\theta]$ is a permanent cycle in the Serre spectral sequence. The only differentials
on $[\theta]$ which could be non-zero are, for dimension reasons, the differentials 
$$d_{k+2}:E^{0,2nm+k}_2=E^{0,2nm+k}_{k+2}\to E^{k+2,2nm-1}_{k+2}$$
which takes values in
$$\ker(\epsilon_\omega^m\cup-: H^{k+2}(G,\pi(\omega)_{\bar{\omega}^m})\to H^{k+2+2nm}(G,\pi(\omega)))\,,$$
respectively 
$$d_{2nm+k+1}: E^{0,2nm+k}_{2nm+k+1}\to E^{2nm+k+1,0}_{2nm+k+1}\,,$$ which takes values in
$$\operatorname{coker}(\epsilon^m_\omega\cup-: H^{k+1}(G,\pi(\omega))\to H^{2nm+k+1}(G,\pi(\omega)_{\omega^m}))\,.$$
Because $k>\ell_0$, we know that
$$\epsilon^m_\omega\cup -: H^{s}(G,M)\to H^{s+2nm}(G,M_{\omega^m})$$
is an isomorphism for $s=k+1$, respectively $s=k+2$, and any $p$-torsion module $M$ of bounded exponent.
The differentials $d_{k+2}$, respectively
$d_{2nm+k+1}$ depicted above are therefore equal to 0. We conclude that the fibrations in the diagram above can be constructed as displayed.
Passing to homotopy limits in the towers $\{F(k)\}_{k\ge 0}$ of that diagram, one obtains a fibration
$$F(\infty):\quad \quad S^{2n-1}_{\Z/p\Z}\to E\to K(G,1)\,,$$ 
as desired. To check that $cd_{\Z/p\Z}(E)<\infty$, one considers the
Serre spectral sequence of the fibration $F(\infty)$ with coefficients in a $\Z/p\Z[G]$-module $L$ and finds that
$H^{j}(E,L)=0$ for $j$ large enough, independent of $L$, finishing the first part of the proof.          

Suppose now conversely that $X$ is a simply connected free $G$-$CW$-complex which is
a $\Z/p\Z$-homology sphere satisfying $cd_{\Z/p\Z}X/G<\infty$.
By Lemma \ref{spacetofibration} there exists a $\Z/p\Z$-spherical fibration
$F\to E\to K(G,1)$ with $F$ simply connected and $cd_{\Z/p\Z}E<\infty$.
Corollary \ref{fibration to twisted} then implies that $G$ has twisted
$p$-periodic cohomology,
completing the proof of Theorem \ref{main}.\hfill$\square$

\section{Algebraic characterization}
Let $x\in H^n(G,\Z/p\Z)$ and consider a $\Z/p\Z [G]$-projective resolution
$$\mathcal{P}_\ast:\quad \cdots\longrightarrow P_n\stackrel{\partial_n}\longrightarrow P_{n-1}
\stackrel{\partial_{n-1}}\longrightarrow\cdots\longrightarrow P_0\longrightarrow\Z/p\Z\,.$$
Put $K_i:=\partial_i{P_i}$. A cocycle representative of $x$ corresponds to a map
$\theta: K_n\to\Z/p\Z$. Form the following diagram

$${\xymatrix{
 &K_n\ar[r]\ar[d]^\theta&P_{n-1}\ar[r]\ar[d]&P_{n-2}\ar[r]\ar[d]^=&\ar@{.}&\ar[r]&P_0\ar[r]\ar[d]^=&\mathbb{Z}/p\mathbb{Z}
 \ar[d]^=\\
 \tilde{x}:&\Z/p\Z\ar[r]&A\ar[r]&P_{n-2}\ar[r]&\ar@{.}[r]&\ar[r]&P_0\ar[r]&\Z/p\Z
}}
$$
where the square on the left is a push-out square. Then the class of the $n$-fold extension
$\tilde{x}$, $[\tilde{x}]\in Ext^n_{\Z/p\Z G}(\Z/p\Z,\Z/p\Z)$, corresponds
to $x\in H^n(G,\Z/p\Z)$.

\begin{lemm}{\label{proj.dim}} Let $e\in H^n(G,\Z/p\Z)$ and consider the $n$-extension
$$\tilde{e}:\quad \Z/p\Z\to A\to P_{n-2}\to\cdots\to P_{0}\to \Z/p\Z$$
as above. Then the following are equivalent.
\begin{itemize}
\item[1)] $G$ has $\Z/p\Z$-periodic cohomology via cup-product with $e$;
\item[2)] the $\Z/p\Z[G]$-projective dimension of $A$ is finite.
\end{itemize}
\end{lemm}
\begin{proof}
Let $\mathcal{P}_\ast$ be a $\Z/p\Z[G]$-projective resolution of $\Z/p\Z$ and choose
$\theta: K_n\to \Z/p\Z$ to represent $e$ as above, giving rise to the $n$-extension $\tilde{e}$.
It is known that cup product with $e$ is induced by a chain map $\Theta:\mathcal{P}_\ast\to
\mathcal{P}_\ast$ of degree $-n$ which extends $\theta$. Consider the following commutative diagram
with exact rows

$${\xymatrix{
 0\ar[r]&K_n\ar[r]\ar[d]^\theta&P_{n-1}\ar[r]\ar[d]^\mu&K_{n-1}\ar[r]\ar[d]^=&0\\
 0\ar[r]&\Z/p\Z\ar[r]&A\ar[r]&K_{n-1}\ar[r]&0\,.
 }}
$$
From the corresponding commutative diagram of long exact $Ext$-sequen\-ces
$${\xymatrix@C-12pt{
\ar[r]& Ext^i_{\Z/p\Z[G]}(K_{n-1},-)\ar[r]& Ext^i_{\Z/p\Z[G]}(P_{n-1},-)\ar[r]&
 Ext^i_{\Z/p\Z[G]}(K_n,-)\ar[r]&\\
 \ar[r]& Ext^i_{\Z/p\Z[G]}(K_{n-1},-)\ar[r]\ar[u]^=&Ext^i_{\Z/p\Z[G]}(A,-)
 \ar[r]\ar[u]&Ext^i_{\Z/p\Z[G]}(\Z/p\Z,-)\ar[r]\ar[u]^{Ext^i_{\Z/p\Z[G]}(\theta,-)}&
}}
$$
follows that $Ext^i_{\Z/p\Z[G]}(\theta,-)$ is an isomorphism for large $i$ if and only
if $proj.dim_{\Z/p\Z[G]}A <\infty$.
\end{proof}

\begin{coro}\label{zeroproj}
Let $G$ be a group with $\Z/p\Z$-periodic cohomology. There exist $k>0$ such
that for all $i\ge k$ and all projective $\Z/p\Z[G]$-modules $P$,
$H^i(G,P)=0$.
\end{coro}
\begin{proof}
By Lemma \ref{proj.dim}, there is a monomorphism $\iota:\Z/p\Z \to A$ with $A$ a
$\Z/p\Z[G]$-module of finite projective dimension $d$ over $\Z/p\Z[G]$. Let $I$ be an injective $\Z/p\Z[G]$-module.
$I$ injects into $A\otimes_{\Z/p\Z}I$ 
via $x\mapsto \iota(1)\otimes x$ and, as $I$ is injective,
$I$ is a retract of $A\otimes_{\Z/p\Z}I$. For any projective
$\Z/p\Z[G]$-module $P$, $P\otimes_{\Z/p\Z} I$ is projective. Thus $proj.dim_{\Z/p\Z[G]}A\otimes_{\Z/p\Z}I\le d$
and, because $I$ is a retract of that module, $proj.dim_{\Z/p\Z[G]}I\le d$ too. We conclude
that the supremum of the the projective length of injective $\Z/p\Z[G]$-modules, $spli\, {\Z/p\Z}[G]$, is $\le d$.
It follows then that $silp\, {\Z/p\Z}[G]$, the supremum of the injective length of projective $\Z/p\Z[G]$-modules, 
is $\le d$  too (see \cite{GG}*{Theorem 2.4}) and we infer that $H^i(G,P)=0$ for $i>d$ and all projective $\Z/p\Z$-modules $P$.
\end{proof}
The following is an algebraic characterization of groups with twisted $p$-periodic cohomology.
  
\begin{lemm}
A group $G$ has twisted $p$-periodic cohomology if and only if
there exists an $n>1$ and an exact sequence
of $\Z/p\Z [G]$-modules 
$$\epsilon:\quad 0\to\Z/p\Z\to A \to P_{n-2}\to\cdots P_0\to \Z/p\Z\to 0$$
with $P_i$ projective for $0\le i\le n-2$ and $proj.dim_{\Z/pZ [G]}(A)<\infty$,
such that $[\epsilon]\in Ext^n_{\Z/p\Z G}(\Z/p\Z,\Z/p\Z)= H^n(G,\Z/p\Z)$ is $\omega$-$p$-integral
(in the sense of Definition \ref{omega-p-integral})
for some $\omega:G\to\hat{\Z}_p^\times$. $G$ has $p$-periodic cohomology if and only if
there is an $\epsilon$ as above with $[\epsilon]$ $p$-integral (for the definition of $\omega$-$p$-integral
cohomology elements see Definition \ref{omega-p-integral}).
\end{lemm}
\begin{proof}
Suppose $G$ has twisted $p$-periodic cohomology. By
definition, there exists
$\sigma: G\to \hat{\Z}^\times_p$ and an an $\sigma$-$p$-integral cohomology class $e_\sigma\in H^m(G,\hat{\Z}_p(\sigma))$
and a $k>0$ such that cup product with $e_\sigma$ induces isomorphisms $H^i(G,M)\to H^{i+m}(G,M_\sigma)$
for all $p$-torsion $\Z[G]$-modules $M$
of finite exponent (we may assume without loss of generality that $m>1$). It follows 
(cf.~Lemma \ref{twistedmodp}) that $G$ has $\Z/p\Z$-periodic cohomology via cup product with
$e:=e_{\sigma}(p)^{p-1}$, where $e_\sigma(p)$ denotes the mod $p$ reduction of $e_\sigma$. Putting $n=m(p-1)$,
it follows that $e$ is $\omega$-$p$-integral with respect to $\omega =\sigma^{p-1}$
and can be represented by an $n$-extension
$$\tilde{e}:\quad 0\to\Z/p\Z\to A \to P_{n-2}\to\cdots \to P_0\to \Z/p\Z\to 0$$
with $P_i$ projective for $0\le i\le n-2$ and $proj.dim_{\Z/pZ [G]}(A)<\infty$. 
Conversely, if we are given an $n$-extension
$$\epsilon:\quad 0\to\Z/p\Z\to A \to P_{n-2}\to\cdots \to P_0\to \Z/p\Z\to 0$$
with $P_i$ projective for $0\le i\le n-2$ and $proj.dim_{\Z/pZ [G]}(A)<\infty$
representing an $\omega$-$p$-integral class $e\in H^n(G,\Z/p\Z)$, then $G$ is twisted
$p$-periodic via cup product with $e$. The untwisted version of the lemma corresponds
to the case where we can choose for $\omega$ the trivial homomorphism.
\end{proof}

\section{Some remarks and examples}{\label{last}}

One cannot expect a group $G$ to have $\Z/p\Z$-periodic cohomology if all its finite subgroups
do (for instance, torsion-free groups do not have $\Z/p\Z$-periodic cohomology in general). 
We will display below a class of groups, for which this assertion holds. For the proofs, we will make use 
of Tate cohomology $\hat{H}^*(G,-)$ for arbitrary groups, as defined in \cite {GM}. 
In case $G$ admits a finite dimensional classifying space for proper actions
$\underbar{EG}$, there is a finitely
convergent stabilizer spectral sequence
$$E_1^{m,n}=\prod_{\sigma\in \Sigma_m}\hat{H}^n(G_\sigma,M)\Longrightarrow \hat{H}^{m+n}(G,M)$$
where $\Sigma_m$ is a set of representatives of $m$-cells of $\underbar{EG}$ and $M$ a $\Z G$-module.
For $G$ a group, $M$ a $\Z G$-module and $\mathcal{F}$ the set of finite
subgroups of $G$, we write
$$\mathcal{H}^q(G,M)\subset \prod_{H\in \mathcal F}\hat{H}^q(H,M)$$
for the set of {\sl{compatible}} families $(u_H)_{H\in\mathcal F}$ with respect to
restriction maps of finite subgroups of $G$, induced by embeddings given by conjugation
by elements of $G$. 

There are many results on groups $G$ which imply the existence of a finite dimensional $\underbar{EG}$.
For instance, groups of cohomological dimension $1$ over $\Q$ do: they act on a tree with finite stabilizers.
Also, if there is a short exact sequence $H\to G\to Q$ of groups and $H$ as well as $Q$ admit a finite dimensional $\underbar{E}$ and
there is a bound on the order of the finite subgroups of $Q$, then there exist a finite dimensional 
model for $\underbar{EG}$
(cf. L\"uck \cite{L}*{Theorem 3.1}).

\begin{lemm}\label{help}
Suppose $G$ admits a finite dimensional $\underbar{EG}$. Then the following holds.
\begin{itemize}
\item[i)] The natural map induced by restricting to finite subgroups
$$\rho: \hat{H}^\ast(G,\Z/p\Z)\to \mathcal{H}^\ast(G,\Z/p\Z)$$
has the property that every element in the kernel of $\rho$ is nilpotent,
and that for every $u\in\mathcal{H}^\ast(G,\Z/p\Z)$ there is a $k$
such that $u^{p^k}$ lies in the image of $\rho$.
\item[ii)] If $\dim\underbar{EG}=t$ and the order of every finite $p$-subgroup of $G$ divides $p^s$
then for every $\Z_{(p)} G$-module $M$ and all $i$
$$p^{s(t+1)}\cdot\hat{H}^i(G,M)=0\,.$$
\item[iii)] If there is a bound on the order of the finite $p$-subgroups of $G$ then
the natural map 
$$\alpha: \hat{H}^\ast(G,\Z_{(p)})\to \hat{H}^\ast(G,\Z/p\Z)$$
has the property that every element in the kernel of $\alpha$ is nilpotent and for any
$u\in \hat{H}^\ast(G,\Z/p\Z)$ there exist $k$ such that $u^{p^k}$ lies
in the image of $\alpha$.

\end{itemize}
\end{lemm}
\begin{proof} i) is Corollary 3.3 of \cite{MT}. 
For ii) we observe that for every $\Z_{(p)}G$-module $M$, the $E_1$-term of the stabilizer
spectral sequence is annihilated by $p^s$. Since $\underbar{EG}$ has dimension $t$, this
implies that $p^{s(t+1)}$ annihilates all groups $\hat{H}^\ast(G,M)$.
For iii) we first use ii) to conclude that $p^{s(t+1)}$ annihilates the 
groups $\hat{H}^*(G,\Z_{(p)})$. One then
argues as in the proof of Lemma 6.6 , Ch.~X in \cite{Br} that for any $\ell>0$ and $x\in \hat{H}^\ast(G,\Z/p\Z),$
$x^{p^\ell}$ lies in the image 
$I_\ell$ of 
$$\hat{H}^\ast(G,\Z/p^{\ell+1}\Z)\to \hat{H}^\ast(G,\Z/p\Z)$$
and that for $\ell$ large enough $I_\ell$ equals the image of the natural map
$$\alpha: \hat{H}^\ast(G,\Z_{(p)})\to \hat{H}^\ast(G,\Z/p\Z)\,,$$ implying one part of iii).
If $y$ lies in the kernel of $\alpha$, the long exact coefficient sequence
associated with the short exact sequence $\Z_{(p)}\stackrel{p}\to \Z_{(p)}\to \Z/p\Z$
shows that $y=pz$ for some $z$ and therefore $y^{s(t+1)}=p^{s(t+1)}z^{s(t+1)}=0$, finishing the proof of iii).
\end{proof}
 
\begin{theo}\label{bound}
Let $G$ be a group which admits a finite dimensional $\underbar{EG}$. Then the following holds.
\begin{itemize}
\item[a)]
$G$ has $\Z/p\Z$-periodic cohomology if and only if all its finite subgroups do;
\item[b)] $G$ has $p$-periodic cohomology if all its finite subgroups do and there is a bound
on the order of the finite $p$-subgroups of $G$.
\end{itemize}
\end{theo}
\begin{proof}
We first prove a). If $G$ has $\Z/p\Z$-periodic cohomology and $e\in H^n(G,\Z/p\Z)$ is a periodicity
generator, then every subgroup $H<G$ has $\Z/p\Z$-periodic cohomology, with periodicity generator
the restriction $e_H\in H^n(H,\Z/p\Z)$. This follows from the natural isomorphism
$H^*(H,M)\cong H^*(G, \operatorname{Coind}_H^G M)$ (Shapiro Lemma). If all finite subgroups
of $G$ have $\Z/p\Z$-periodic cohomology, there exist a unit in $e\in \hat{H}^n(G,\Z/p\Z)$ for some
$n>0$ (cf.~\cite{MT}*{Theorem 4.4}). Since $\dim{\underbar{EG}}$ is finite, there is a $k>0$ such that
the natural map $\theta: H^j(G,M)\to \hat{H}^j(G,M)$ is an isomorphism for all $j\ge k$ and all $\Z G$-modules $M$.
Choose $\ell$ such that the degree of $e^\ell$ is larger than $k$ and choose $\epsilon\in H^{n\ell}(G,\Z/p\Z)$
such that $\theta(\epsilon)=e^\ell$. Then $G$ has $\Z/p\Z$-periodic cohomology with periodicity generator $\epsilon$,
finishing the proof of a). 
For b) we assume that all finite subgroups of $G$ have $p$-periodic cohomology and that there is a 
bound on the order of the finite $p$-subgroups. From Theorem 4.4 of \cite{MT} we conclude that there exist
a unit $u\in \hat{H}^n(G,\Z/p\Z)$ for some $n>0$. Let $v$ an inverse for $u$. By Lemma \ref{help} we can find $k>0$
and $\tilde{u},\tilde{v}\in \hat{H}^\ast(G,\Z_{(p)})$ such that $\alpha(\tilde{u})=u^{p^k}$ and
$\alpha(\tilde{v})=v^{p^k}$, where $\alpha: \hat{H}^\ast(G,\Z_{(p)})\to \hat{H}^\ast(G,\Z/p\Z)$ is the natural
map. From Lemma \ref{help} we conclude that $1-\tilde{u}\tilde{v}$ is nilpotent, thus $\tilde{u}\tilde{v}$ is
invertible, and we conclude that $\tilde{u}\in \hat{H}^{np^k}(G,\Z_{(p)})$ is a unit. Since $G$ admits a
finite dimensional $\underbar{EG}$, the supremum of the injective length of projective
$\Z G$-modules, $silp\, \Z G$, is finite. Therefore, 
there is an $n_0$ such that $H^n(G,P)=0$ for all $n>n_0$ and all projective
$\Z G$-modules $P$. By a basic property of Tate cohomology, this implies that there exist $m>0$ such that the canonical map
$\lambda: H^i(G,L)\to \hat{H}^i(G,L)$ is an isomorphism for all $i>m$ and all $\Z G$-modules $L$.
It follows that by choosing an $r>0$ such that $\tilde{u}^r$ has degree larger than $m$, that there is an $\epsilon\in
H^{np^kr}(G,\Z_{(p)})$ with $\lambda(\epsilon)=\tilde{u}^r$. Let
$$\beta: H^{np^kr}(G,\Z_{(p)})\to H^{np^kr}(G,\hat{\Z}_p)$$ be the canonical map and
put $e=\beta(\epsilon)$. Then cup-product with $e$ induces isomorphisms
$$e\cup - : H^j(G,M)\to H^{j+np^kr}(G,\hat{\Z}_p\otimes M)=H^{j+np^kr}(G,M)$$
for all $j>m$ and all $p$-torsion $\Z G$-modules $M$ of bounded exponent, proving that $G$ has $p$-periodic
cohomology.
\end{proof}

Note that we made use of the bound condition in b) of
Theorem \ref{bound} to prove the result, but that bound is not a necessary condition.
For instance, $\Z_{p^\infty}$ has $p$-periodic cohomology, but no bound on the order
of its finite $p$-subgroups. On the other hand, the following is an example
of a group $G$ which admits a finite dimensional $\underbar{EG}$
and with all finite subgroups having $p$-periodic cohomology, but $G$ not having $p$-periodic
cohomology. 
Let $\alpha\in \hat{\Z}_p^\times$ be a $p$-adic
unit and define $G(\alpha)$ to be the semi-direct product  $\Z_{p^\infty}\rtimes_\alpha \Z$,
where we have identified $Aut(\Z_{p^\infty})$ with $\hat{\Z}_p^\times$.

\begin{exam}\label{mainexample}
Let p be an odd prime and put $G(1+p)=\Z_{p^\infty}\rtimes_{1+p}\Z$.
\begin{itemize}
\item[a)] 
$G(1+p)$ has $\Z/p\Z$-periodic cohomology
of period 2.
\item[b)] $G(1+p)$ does not have $p$-periodic cohomology.
\item[c)] $G(1+p)$ has twisted $p$-periodic cohomology.
\item[d)] $G(1+p)$ acts freely on a simply connected 7-dimensional\\ $G(1+p)$-$CW$-complex 
which is a $\Z/p\Z$-homology $3$-sphere.
\end{itemize}
\end{exam}

\begin{proof}
a): Let $H=\Q\rtimes_{1+p}\Z$. There is a natural surjective map $H\to G(1+p)$
with kernel isomorphic to $\Z_{(p)}$. Note that $H$ has cohomological dimension 3. Choose $Y$ to be a
3-dimensionl model for $K(H,1)$ and $X$ the covering space corresponding to $\Z_{(p)}<H$.
$X$ is a free $G(1+p)$-$CW$-complex and $X\simeq K(\Z_{(p)},1)$, thus $X$ is a $\Z/p\Z$-homology $1$-sphere.
We then have homotopy fibration
$$X\to X/G(1+p)\to BG(1+p)\,,\quad X_{\Z/p\Z}= S^1_{\Z/p\Z}$$
which is $H\Z/p\Z$-orientable, because multiplication by $1+p$
is the identity on $\Z/p\Z$. It follows that the associated $\Z/p\Z$-Euler class $e\in H^2(G(1+p),\Z/p\Z)$ induces
via cup product isomorphisms 
$$H^i(G(1+p),M)\stackrel{e\cup -}\longrightarrow H^{i+2}(G(1+p),M)$$
for all $i>3$ and all $\Z/p\Z[G]$-modules
$M$, which proves a). 
For b) we consider the subgroups 
$$G_n=\Z/p^n\Z\rtimes_{1+p}\Z<G(1+p)$$
and observe that the minimal $p$-period
for $H^*(G_n,\Z/p^n\Z)$ is $\ge 2p^{n-1}$, because multiplication by $1+p$ on $H^2(\Z/p^n\Z,\Z/p^n\Z)=\Z/p^n\Z$ is an
automorphism of order $p^{n-1}$ for odd $p$. Thus, the minimal $p$-period for $G_n$ goes to $\infty$ as $n$ tends to $\infty$ and, therefore, $G(1+p)$ does not have $p$-periodic
cohomology. For c) we observe that the twisted $\hat{\Z}_p$-Euler class $\tilde{e}\in H^2(G(1+p),\hat{\Z}_p(\omega))$
of the homotopy $\Z/p\Z$-spherical fibration constructed in a), with $\omega: G(1+p)\to\hat{\Z}_p^\times$
given by $(x,y)\mapsto (1+p)^y$ for $(x,y)\in\Z_{p^\infty}\rtimes \Z$, has reduction mod $p$ equal to the
$\Z/p\Z$-Euler class $e$ of a). It follows that $G(1+p)$ has twisted
$p$-periodic cohomology of period $2$, with twisted $p$-periodicity induced by cup product with $\tilde{e}$.
For d) we again look at the free $G(1+p)$-$CW$-complex $X$ as constructed in a). The join $X\ast X$ is a simply connected free $G(1+p)$-$CW$-complex of dimension $7$, which is a 
$\Z/p\Z$-homology $3$-sphere, completing the proof.
\end{proof}

\end{document}